\documentclass[reqno,twoside]{amsart}
\usepackage{comment}
\usepackage{etex}
\usepackage{amscd}
\usepackage{amsfonts}
\usepackage{amsmath}
\usepackage{amssymb}
\usepackage{amsthm}
\usepackage{amsaddr}
\usepackage{fancyhdr}
\usepackage{latexsym}
\usepackage[colorlinks=true, pdfstartview=FitV, linkcolor=blue, citecolor=blue, urlcolor=blue]{hyperref}
\usepackage{enumitem}      
\usepackage{mathtools}            
\usepackage{indentfirst} 
\usepackage{color}
\usepackage{caption}          
\usepackage[normalem]{ulem}
\usepackage{upgreek}
\usepackage{dutchcal}
\usepackage{mathrsfs}
\renewcommand{\b}{\textcolor{blue}}

\usepackage{subcaption}
\newcommand{\R}{\mathbb{R}}

\newcommand{\eq}[1]{\begin{align}#1\end{align}}
\newcommand{\eqs}[1]{\begin{align*}#1\end{align*}}

	\newcommand{\p}[1]{\left( #1 \right)}

	\newcommand{\norm}[1]{\left\lVert #1 \right\rVert}

	\newcommand{\alignlabel}[1]{\stepcounter{equation} \tag{\theequation} \label{#1}}

	\let\EPSILON\epsilon
	\let\VAREPSILON\varepsilon

	\renewcommand{\epsilon}{\VAREPSILON}
	\renewcommand{\varepsilon}{\EPSILON}
	
	\renewcommand{\subset}{\subseteq}

	\renewcommand{\b}[1]{\left[ #1 \right]}

	\renewcommand{\hat}{\widehat}
	\renewcommand{\tilde}{\widetilde}

	\definecolor{red}{RGB}{220,0,0}
	\definecolor{green}{RGB}{0,180,0}
	\usepackage{thmtools}
	\declaretheoremstyle[headfont=\kpfonts]{normalhead}
	
	\newtheorem{theorem}{Theorem}[section]
	
	\newtheorem{lemma}{Lemma}[section]

	\theoremstyle{definition}
	
	\newtheorem{remark}{Remark}[section]
	
	\allowdisplaybreaks
	\numberwithin{equation}{section}
	\numberwithin{figure}{section}
	\usepackage{geometry}
	\usepackage{tikz}
	\usetikzlibrary{arrows.meta}
	\usetikzlibrary{decorations.markings}
	\tikzset{->-/.style={decoration={
				markings,
				mark=at position #1 with {\arrow{>}}},postaction={decorate}}}
	\tikzset{middlearrow/.style={
			decoration={markings,
				mark= at position 0.55 with {\arrow{#1}} ,
			},
			postaction={decorate}
		}
	}
	\let\OLDthebibliography\thebibliography
	\renewcommand\thebibliography[1]{
		\OLDthebibliography{#1}
		\setlength{\parskip}{0pt}
		\setlength{\itemsep}{0pt plus 0.3ex}
	}
	\AtBeginDocument{
		\def\MR#1{}
	}

	\makeatletter
	\renewcommand\subsection{\@startsection{subsection}{2}%
		\z@{-0.8\linespacing\@plus-0.7\linespacing}{0.7\linespacing}%
		{\bfseries}}
	\makeatother

	\makeatletter
	\@namedef{subjclassname@2020}{%
		\textup{2020} Mathematics Subject Classification}
	\makeatother
	
	\usetikzlibrary{decorations.markings}
	
\begin{document}
		
		\title{Local well-posedness of 
        the higher-order nonlinear Schr\"{o}dinger equation on the half-line: the case of two boundary conditions}
		
		\author{Aykut Alkın$^\dagger$, Dionyssios Mantzavinos$^\ddagger$,  Türker Özsarı$^{*}$}
		
		\address{\normalfont $^{\dagger}$Department of Mathematics, Izmir Institute of Technology\\
			\normalfont $^{\ddagger}$Department of Mathematics, University of Kansas
			\\
			\normalfont $^{*}$Department of Mathematics, Bilkent University
		} \email{\!turker.ozsari@bilkent.edu.tr \textnormal{(corresponding author)}}

		\thanks{\textit{Acknowledgements.} DM gratefully acknowledges support from the U.S. National Science
Foundation (NSF-DMS 2206270 and NSF-DMS 2509146) and the Simons Foundation (SFI-MPS-TSM-00013970). 
		}
		%
		
		%
		\begin{abstract}
	We prove the local Hadamard well-posedness of the higher-order nonlinear Schrödinger equation with negative dispersion coefficient on the half-line under nonzero Dirichlet and Neumann boundary data. The paper uncovers a striking feature of this model: unlike the positive-dispersion case studied in \cite{amo2024}, well-posedness requires two boundary conditions rather than one. This reflects a structural limitation in the underlying global relation, which allows elimination of at most one unknown boundary trace. Using the unified transform of Fokas, we derive an explicit solution formula for the associated forced linear problem and establish sharp Sobolev and Strichartz estimates on the half-line. These linear estimates, together with corresponding Cauchy problem bounds, yield local well-posedness in both high- and low-regularity Sobolev spaces, as well as local Lipschitz continuity of the solution map.
		\end{abstract}
		
		\vspace*{-0.5cm}
		
		\maketitle
		\markboth
		{A. Alkın, D. Mantzavinos \& T. Özsarı}
		{Local well-posedness of the higher-order nonlinear Schr\"odinger equation on the half-line}
		

		\section{Introduction}
		
		We consider the higher-order nonlinear Schr\"odinger (HNLS) equation with a power nonlinearity on the half-line $(0, \infty)$ with nonzero Dirichlet-Neumann boundary conditions:
		\begin{equation}
			\begin{aligned}\label{nonlinear}
				&i u_t + i \beta u_{xxx} + \alpha u_{xx} + i \delta u_x = \kappa |u|^{\lambda - 1} u, \quad 0 < x < \infty, \ 0 < t < T,
				\\
				&u(x, 0) = u_0 (x) \in H^s(0, \infty),\\
				&u(0, t) = g_0(t) \in H^{\frac{s+1}{3}}(0, T), \quad u_x (0, t) =g_1(t) \in H^{\frac{s}{3}}(0, T),
			\end{aligned}
		\end{equation}
		where $\beta < 0$, $\alpha, \delta \in \mathbb R$, $\kappa \in \mathbb C$, $\lambda>1$, and $T>0$ is an appropriate lifespan to be determined. 
		In the above model, $H^s(0, \infty)$ is defined as the restriction on the half-line interval $(0, \infty)$ of the usual Sobolev space on the whole line 
		\eq{
			H^s(\R) := \left\{\phi \in L^2(\R): \p{1+k^2}^{\frac s2} \mathcal F\{\phi\} \in L^2(\R)\right\}, \quad s\geq 0,
		}
		where $\mathcal F\{\cdot\}$ denotes the Fourier transform. An equivalent characterization of $H^s(0, \infty)$ is the Sobolev space $W^{s, 2}(0, \infty)$ of functions in $L^2(0, \infty)$ whose first $s$ derivatives belong in $L^2(0, \infty)$ when $s$ is a non-negative integer, and the definition extends to any $s \geq 0$ via interpolation or explicitly through the Sobolev-Slobodeckij norm.
		The Sobolev spaces $H^{\frac{s+1}{3}}(0, T)$ and $H^{\frac{s}{3}}(0, T)$ for the two pieces of boundary data are defined as the restrictions on the finite interval $(0, T)$ of the Sobolev spaces on the whole line with the corresponding indices. It should be noted that the two Sobolev exponents $\frac{s+1}{3}$ and $\frac{s}{3}$ for the boundary data depend on the Sobolev exponent $s$ for the initial data.

The higher-order nonlinear Schr\"odinger (HNLS) equation may be viewed as a refinement of the classical nonlinear Schr\"odinger model in regimes where higher-order dispersive effects can no longer be neglected. In particular, the third-order term plays an essential role in the description of ultrashort pulse propagation in nonlinear optical fibers; see, for example, \cite{koda85, koda87}. In its physically derived form, the model may also involve derivative nonlinear terms, but in the present work we focus on a pure power nonlinearity of any superlinear order, thereby treating a direct higher-dispersion analogue of NLS while retaining a sufficiently flexible nonlinear framework. From a mathematical perspective, this places the equation at an interface between Schr\"odinger-type and Korteweg-de Vries-type dispersive behavior.

The Cauchy problem for HNLS on the whole line has been studied in several works, including \cite{Laurey97,Tak00,Carvajal04,Carvajal06,Fam22}. In contrast, the corresponding nonhomogeneous initial-boundary value problems are much less developed. In our previous work \cite{amo2024}, we established local Hadamard well-posedness on the half-line in the case $\beta>0$, where only one boundary condition is required. In the same work, it was remarked that the case $\beta<0$ would require two boundary conditions at $x=0$, which was subsequently also verified in the recent papers~\cite{hy2025,hy2026} in a particular context restricted to a monomial linear part. Moreover, in \cite{MMO26} and \cite{OK22}, the finite-interval problem with mixed Dirichlet-Neumann boundary data was treated. The present paper addresses the genuinely different half-line configuration $\beta<0$, in which both Dirichlet and Neumann data must be prescribed at the boundary point $x=0$. 

A noteworthy feature of the problem studied here is that the need for two boundary conditions is not merely a formal consequence of the negative sign of the highest-order term itself (e.g. in the case of the heat equation only one boundary condition is needed) but rather it is closely tied to the overall spectral structure of the associated linear problem. Indeed, after deriving a certain spectral identity known as the global relation for the forced linear equation, one finds that the relevant spectral symmetry allows the elimination of at most one unknown boundary trace. As a result, a single prescribed boundary value is not sufficient for closing the representation formula when $\beta<0$. This sharply contrasts with the positive-dispersion half-line problem considered in \cite{amo2024}.

Our main results stated below show that the negative-dispersion HNLS half-line problem admits a complete Hadamard well-posedness theory despite the genuinely new structural difficulty created by the need to prescribe two boundary conditions at $x=0$.  

\begin{theorem}[High regularity well-posedness]\label{HighRegThem}
    Let $\frac12<s\le 2$, $s\neq \frac32$, $\lambda>1$, where if $\lambda\not\in 2\mathbb{N}+1$, then assume 
    \[
\text{if } s \in \mathbb{N}, \text{ then } \lambda \ge s+1 \text{ if } \lambda \in 2\mathbb{N}; \ \lfloor \lambda \rfloor \ge s \text{ if } \lambda \notin \mathbb{N},
\]
\[
\text{if } s \notin \mathbb{N}, \text{ then } \lambda > s+1 \text{ if } \lambda \in 2\mathbb{N}; \ \lfloor \lambda \rfloor \ge \lfloor s \rfloor +1 \text{ if } \lambda \notin \mathbb{N}.
\] Let also $u_0\in H^s_x(\mathbb{R_+})$, $g_0\in H^{\frac{s+1}{3}}_{\text{loc},x}(\mathbb{R_+})$, $g_1\in H^{\frac{s}{3}}_{\text{loc},x}(\mathbb{R_+})$ and assume that the compatibility conditions
$$u_0(0)=g_0(0), \, s>\frac12; \quad u_0'(0)=g_1(0), \, s>\frac32$$ hold.  Then, there exists $T>0$ such that \eqref{nonlinear} admits a unique solution in $u\in C([0,T];H^s_x(\mathbb{R}_+))$. Moreover, the data-to-solution map $(u_0,g_0,g_1)\mapsto u$ is locally Lipschitz continuous. 

\end{theorem}

\begin{theorem}[Low regularity wellposedness]\label{LowRegThem} Let $0\le s<\frac12$, $2\le \lambda\le \frac{7-2s}{1-2s}$, $
\mu=\frac{6\lambda}{(1-2s)(\lambda-1)}$, and $ r=\frac{2\lambda}{1+2(\lambda-1)s}
$. Let also $u_0\in H^s_x(\mathbb{R_+})$, $g_0\in H^{\frac{s+1}{3}}_{\text{loc},x}(\mathbb{R_+})$, $g_1\in H^{\frac{s}{3}}_{\text{loc},x}(\mathbb{R_+})$.  If $\lambda=\frac{7-2s}{1-2s}$, further assume that $u_0$ is small. Then,~\eqref{nonlinear} admits a unique solution in $u\in C([0,T];H^s_x(\mathbb{R}_+))\cap L^\mu_t(0,T;H^{s,r}_x(\mathbb{R}_+))$. Moreover, the data-to-solution map $(u_0,g_0,g_1)\mapsto u$ is locally Lipschitz continuous. 
\end{theorem}

The proofs of Theorems \ref{HighRegThem} and \ref{LowRegThem} are provided in Section \ref{Sec_contraction}. The corresponding analysis relies on the unified transform of Fokas \cite{Fokas97,Fokas08}, used here as the boundary-value analogue of the Fourier transform. The starting point is an explicit solution formula for the associated forced linear problem on the half-line:
	\begin{equation}
		\begin{aligned}\label{linear}
			&i u_t + i \beta u_{xxx} + \alpha u_{xx} + i \delta u_x = f(x, t), \quad 0 < x < \infty, \  0 < t < T,
			\\
			&u(x, 0) = u_0 (x),\\
			&u(0, t) = g_0(t), \quad u_x (0, t) = g_1(t),
		\end{aligned}
	\end{equation}
	where $f(x, t)$ is a given forcing. The implementation of the unified transform in the present setting requires a careful treatment of the multi-term dispersive symbol, nontrivial spectral maps involving complex square roots, and contour deformations necessitated by the associated branch-cut structure. 
    
    With an explicit linear solution formula at hand, the subsequent task is to derive the Sobolev and Strichartz estimates  for the forced linear problem \eqref{linear} needed for the nonlinear iteration that leads to the proof of Theorems~\ref{HighRegThem} and~\ref{LowRegThem}. For this purpose, we adapt the method introduced in \cite{fhm2017,fhm2016} for the nonlinear Schr\"odinger and Korteweg-de Vries equations on the half-line, which has since been employed for the rigorous well-posedness of several other nonlinear initial-boundary value problems, including \cite{oy2019,hm2020,mo2025}. The nonlinear problem \eqref{nonlinear} is then handled by a contraction argument built on these linear estimates. In the high-regularity regime of Theorem  \ref{HighRegThem}, the Sobolev algebra property is sufficient to control the power nonlinearity, whereas in the low-regularity regime of Theorem \ref{LowRegThem} one must instead rely on Strichartz estimates in suitable Bessel potential spaces. In this way, we obtain local Hadamard well-posedness for both smooth and rough data, together with local Lipschitz continuity of the data-to-solution map. 
    
    We remark that the results of Theorems~\ref{HighRegThem} and~\ref{LowRegThem} cover a wide range of integer as well as non-integer nonlinearities as specified above, thanks to the use of Strichartz estimates. This is in contrast with approaches for other related equations via the use of Bourgain spaces, such as the recent works~\cite{hy2025,hy2026}, which are by nature restricted to specific integer nonlinearities of quadratic and cubic order, respectively.

We conclude by emphasizing that  the study of nonlinear evolution equations in the setting of nonhomogeneous initial-boundary value problems presents several additional challenges when compared to the more standard initial value (Cauchy) problem on the whole line. Importantly, the Fourier transform (and associated harmonic analysis toolbox) is no longer available when the spatial domain involves a boundary. Besides the method of~\cite{fhm2017,fhm2016} employed in this work, which circumvents the lack of Fourier transform via the unified transform of Fokas,  two other main approaches are the temporal Laplace transform method of \cite{bsz2002}, which has since been used in several other works such as \cite{bsz2008,kai2013,ozs2015,BO16,et2016},
and  the  boundary forcing operator method of \cite{ck2002, h2005,h2006}, recently also employed in \cite{c2017,cc2020}. 
Other noteworthy contributions on the rigorous analysis of nonlinear dispersive initial-boundary value problems are those by Faminskii~\cite{fam2004,fam2007,f2020,f2024}. 
\\[2mm]
\textit{Structure.} The derivation of the unified transform solution formula for the forced linear problem \eqref{linear} is given in Section \ref{app-s}. Sobolev  linear estimates for a certain reduced linear problem on the half-line are established in Section~\ref{red-s}.  Section \ref{ivp-s} provides a summary of existing Sobolev and Strichartz estimates for the Cauchy problem, while Strichartz estimates for the reduced half-line problem are derived in Section \ref{str-s}. Finally, the analysis leading to the well-posedness of the nonlinear problem \eqref{nonlinear} as stated in Theorems \ref{HighRegThem} and \ref{LowRegThem} is presented in Section \ref{Sec_contraction}.

\section{Linear solution formula via the unified transform}\label{app-s}

We employ the unified transform of Fokas to derive the solution formula for the forced linear problem on the half-line given by \eqref{linear}. Throughout our derivation, we work under the assumption of sufficient smoothness and decay. 
We begin by defining the half-line Fourier transform pair
\eq{\label{fi-ft-def}
	\hat \phi (k) := \int_0^\infty e^{-i k x} \phi (x) dx, \quad \text{Im}(k)\le 0,
	\qquad
	\phi (x) =
	\frac{1}{2 \pi} \int_{-\infty}^\infty e^{i k x} \hat \phi (k) dk, \quad  0 < x < \infty.
}
Taking the half-line Fourier transform \eqref{fi-ft-def} of the forced linear higher-order Schr\"odinger equation in \eqref{linear} and then integrating in $t$, we obtain 
\begin{equation}\label{globalrelation}
	e^{\omega t}\hat{u}(k,t)=\hat{u}_0(k)- i \int_0^t e^{\omega t'} \hat f(k, t') dt'+(-\beta k^2+\alpha k+\delta)\tilde{g}_0(\omega,t)+(i\beta k-i\alpha)\tilde{g}_1(\omega,t)+\beta\tilde{g}_2(\omega,t),\quad \text{Im}(k)\leq 0,
\end{equation}
where  
\eq{\omega:=-i\beta k^3+i \alpha k^2+i\delta k,\quad
	g_2 (t) = u_{xx} (0, t), 
	\quad 
	\tilde \phi (\omega, t) := \int_0^t e^{\omega t'} \phi (t') dt'. \label{tilde-transform}}
Inverting the global relation for $k\in\R$ via \eqref{fi-ft-def}, we obtain the following integral representation for the solution:
\eqs{
	\begin{split}
		u(x, t) &= \frac{1}{2 \pi} \int_{-\infty}^{\infty} e^{i k x - \omega t} \b{\hat u_0 (k) - i \int_0^t e^{\omega t'} \hat f(k, t') dt'} dk\\
		&\quad + \frac{1}{2 \pi} \int_{-\infty}^{\infty} e^{i k x -\omega t} \Big[\beta \tilde g_2 (\omega, t) + i \p{\beta k - \alpha} \tilde g_1 (\omega, t) - \p{\beta k^2 - \alpha k - \delta} \tilde g_0 (\omega, t)\Big] dk.
	\end{split} \alignlabel{intrep}}
This expression is not an explicit formula for the solution because it contains the unknown boundary value $g_2$ through its time transform defined by \eqref{tilde-transform}.  To transform \eqref{intrep} into an explicit formula, we will follow different strategies depending on the sign of the quantity $\alpha^2+3\beta\delta$ dictated by the constant parameters of the main equation.

We introduce the region $D^+:=D\cap\mathbb{C}^+$, where $D:=\{ k\in\mathbb{C}: \; \text{Re}\;\omega(k)<0 \}$ and $\mathbb{C}^+:=\{k\in\mathbb{C}: \;\text{Im }k>0\}$. A more explicit description of $D^+$ is as follows:
\begin{eqnarray}\label{Dplus}
	D^+=\left\{ k=k_R+ik_I\in\mathbb{C}^+\; : \quad3\left(k_R -\frac{\alpha}{3\beta}\right)^2-k_I^2-\frac{\alpha^2+3\beta\delta}{3\beta^2}>0 \right\},
\end{eqnarray}
where the real and imaginary parts of $k$ are denoted by $k_R$ and $k_I$, respectively. It should be noted that $\partial D^+\subseteq\{k\in \overline{\mathbb{C}^+}\,|\,\text{Re}\,{\omega(k)}=0\}$ (the equality does not always hold), see Figure \ref{figure1}.

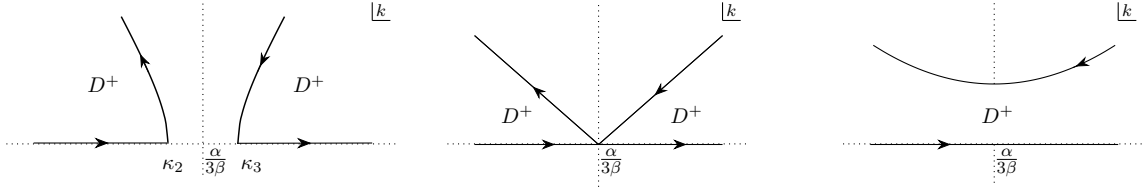
\begin{figure}[ht]
	\centering
	\vspace{2.5cm}
		\begin{tikzpicture}[scale=0.8]
		\pgflowlevelsynccm
		\draw[line width=.5pt, black, dotted](-3.25,0.01)--(3.25,0.01);
		\draw[line width=.5pt, black, dotted](0,-0.5)--(0, 2.35);
		\draw[line width=.5pt, black](2.8,2.35)--(2.8,2.05);
		\draw[line width=.5pt, black](2.8,2.05)--(3.1,2.05);
		\node[] at (2.89, 2.225) {\fontsize{8}{8} $k$};
		\node[] at (-0.55, -0.35) {\fontsize{10}{10} $\kappa_2$};
		\node[] at (0.75, -0.35) {\fontsize{10}{10} $\kappa_3$};
		\node[] at (-1.7, 1) {\fontsize{10}{10} $D^+$};
		\node[] at (1.7, 1) {\fontsize{10}{10} $D^+$};
		\node[] at (0.15, -0.3) {\fontsize{11}{11} $\frac{\alpha}{3\beta}$};
		\draw[middlearrow={Stealth[scale=1.3, reversed]}, black] (-1,1.41) -- (-53:-1.5);
		\draw[middlearrow={Stealth[scale=1.3]}, black] (1,1.41) -- (53:1.5);
		\draw[middlearrow={Stealth[scale=1.3, reversed]}, black, line width=.5pt] (-0.565,0.025) -- (-0.5:-2.8);
		\draw[middlearrow={Stealth[scale=1.3]}, black, line width=.5pt] (0.565,0.02) -- (0.5:2.8);
		\draw[domain=0.5774:1.35, variable=\x, smooth, black, line width=0.7pt] plot ({\x}, {sqrt(3*\x*\x-1)});
		\draw[domain=-1.35:-0.5774, variable=\x, smooth, black, line width=0.7pt] plot ({\x}, {sqrt(3*\x*\x-1)});
	\end{tikzpicture}
	\hspace*{5cm}
	\begin{tikzpicture}[scale=0.8]
	\pgflowlevelsynccm
	\draw[line width=.5pt, black, dotted](-2.5,0.01)--(2.5,0.01);
	\draw[line width=.5pt, black, dotted](0,-0.7)--(0, 2.35);
	\draw[line width=.5pt, black](-2.05,0)--(2.05,0);
	\draw[line width=.5pt, black](-2.05,1.8)--(0,0);
	\draw[line width=.5pt, black](2.05,1.8)--(0,0);
	\draw[line width=.5pt, black](2.05,2.35)--(2.05,2.05);
	\draw[line width=.5pt, black](2.05,2.05)--(2.35,2.05);
	\node[] at (2.14, 2.225) {\fontsize{8}{8} $k$};
	\node[] at (0.15, -0.3) {\fontsize{11}{11} $\frac{\alpha}{3\beta}$};
	\node[] at (-1.4, 0.5) {\fontsize{10}{10} $D^+$};
	\node[] at (1.4, 0.5) {\fontsize{10}{10} $D^+$};
	\draw[middlearrow={Stealth[scale=1.3]}, black] (-1.6,0) -- (-1.5:0);
	\draw[middlearrow={Stealth[scale=1.3]}, black] (1.4,0) -- (1.5,0);
	\draw[middlearrow={Stealth[scale=1.3]}, black] (2.05,1.8) -- (1.5:0);
	\draw[middlearrow={Stealth[scale=1.3,reversed]}, black] (-2.05,1.8) -- (1.5:0);
\end{tikzpicture}
	\hspace*{5cm}
\begin{tikzpicture}[scale=0.8]
	\pgflowlevelsynccm
	\draw[line width=.5pt, black, dotted](-2.5,0.01)--(2.5,0.01);
	\draw[line width=.5pt, black, dotted](0,1)--(0, 2.35);
	\draw[line width=.5pt, black, dotted](0,0)--(0, -0.7);
	\draw[line width=.5pt, black](-2.05,0)--(2.05,0);
	\draw[line width=.5pt, black](2.05,2.35)--(2.05,2.05);
	\draw[line width=.5pt, black](2.05,2.05)--(2.35,2.05);
	\node[] at (2.14, 2.225) {\fontsize{8}{8} $k$};
	\node[] at (0.15, -0.3) {\fontsize{11}{11} $\frac{\alpha}{3\beta}$};
	\node[] at (0, 0.5) {\fontsize{10}{10} $D^+$};
	\draw[middlearrow={Stealth[scale=1.3]}, black] (-1.6,0) -- (1.6:0);
	\draw[domain=-2.0:2, variable=\x, smooth, black, line width=0.5pt] plot ({\x}, {(0.4*(\x))^2+1});
	\draw[middlearrow={Stealth[scale=1.3]}, black] (1.4,1.3136) -- (1.3,1.2704);
\end{tikzpicture}
	\vspace{7mm}
	\caption{The marked open regions show $D^+$ defined by \eqref{Dplus} for $\alpha^2+3\beta\delta>0$ (left), $\alpha^2+3\beta\delta=0$ (center), and $\alpha^2+3\beta\delta<0$ (right). In the left figure, $\kappa_2=\frac{\alpha+\sqrt{\alpha^2+3\beta\delta}}{3\beta}$ and  $\kappa_3=\frac{\alpha-\sqrt{\alpha^2+3\beta\delta}}{3\beta}$.}
	\label{figure1}
\end{figure}

Using {standard complex analytical arguments}, we rewrite \eqref{intrep} as follows:
\eqs{
	\begin{split}
		u(x, t) &= \frac{1}{2 \pi} \int_{-\infty}^{\infty} e^{i k x - \omega t} \b{\hat u_0 (k) - i \int_0^t e^{\omega t'} \hat f(k, t') dt'} dk\\
		&\quad + \frac{1}{2 \pi} \int_{\partial D^+} e^{i k x -\omega t} \Big[\beta \tilde g_2 (\omega, t) + i \p{\beta k - \alpha} \tilde g_1 (\omega, t) - \p{\beta k^2 - \alpha k - \delta} \tilde g_0 (\omega, t)\Big] dk.
	\end{split} \alignlabel{intrep_D+}}

In order to eliminate the unknown transform from \eqref{intrep_D+}, we employ the global relation \eqref{globalrelation} and replace $k$ with a nontrivial transformation of $k$, say $\nu=\nu(k)$ satisfying: 
\begin{itemize} 
	\item[(i)] the invariance property $\omega(k)=\omega(\nu(k)),$ $k\in \overline{D^+}$ and also 
	\item[(ii)] the range condition $\text{Im}(\nu(k))\le 0$ for $k\in \overline{D^+}$.
\end{itemize}
A naive attempt to find a solution to the invariance property (i) under the assumption that $\nu$ is not identity map, suggests analysis of the quadratic equation 
\begin{equation}\label{second}
	\nu^2+\left(k-\frac{\alpha}{\beta}\right)\nu+\left(k^2-\frac{\alpha}{\beta}k-\frac{\delta}{\beta}\right)=0.
\end{equation} This leads to two solutions $\nu_{\pm}$ given by
\begin{equation}\label{maps}
	\nu_{\pm}(k)=-\frac{1}{2}\left(k-\frac{\alpha}{\beta}\right)\pm i\frac{\sqrt{3}}{{2}}\left(\big( k-\frac{\alpha}{3\beta} \big)^2-\frac{4(\alpha^2+3\beta\delta)}{9\beta^2}\right)^{\frac12},
\end{equation} where the square root is defined by $z^{\frac12}=|z|e^{iArg(z)},\, Arg(z)\in (-\pi,\pi]$.
It will be shown that with this definition of the square root, one can guarantee that exactly one of $\nu_\pm$ satisfies the \emph{range condition}.  To this end, if $\alpha^2+3\beta\delta >0$, the branch cut of $\nu_\pm$ is the finite segment on the real line connecting (the branch points) $\kappa_1 := \frac{\alpha+2\sqrt{\alpha^2+3\beta\delta}}{3\beta}$ and $\kappa_4:= \frac{\alpha-2\sqrt{\alpha^2+3\beta\delta}}{3\beta}$, together with the vertical line $\text{Re}\, k=\frac{\alpha}{3\beta}$.  If $\alpha^2+3\beta\delta =0$,  the branch cut is the vertical line $\text{Re}\, k=\frac{\alpha}{3\beta}$.  Finally, if $\alpha^2+3\beta\delta <0$, the branch cut of $\nu_\pm$ is the union of two rays lying on the axis $k_R=\frac{\alpha}{3\beta}$, emanating from the branch points $\kappa_5=\frac{\alpha}{3\beta}-i\frac{2\sqrt{-(\alpha^2+3\beta\delta)}}{3\beta}$, $\kappa_6=\frac{\alpha}{3\beta}+i\frac{2\sqrt{-(\alpha^2+3\beta\delta)}}{3\beta}$, respectively, and extending to infinity. 

In the case $\alpha^2+3\beta\delta \ge 0$, before we eliminate the unknown $\tilde{g_2}(\omega,t)$ from the solution formula,  we first deform the contour of integration to a more appropriate contour as in Figure \ref{figure2} which avoids the branch cut.  This is not performed in the case $\alpha^2+3\beta\delta < 0$ because in this setting the branch cut is strictly away from $\overline{D^+}$.

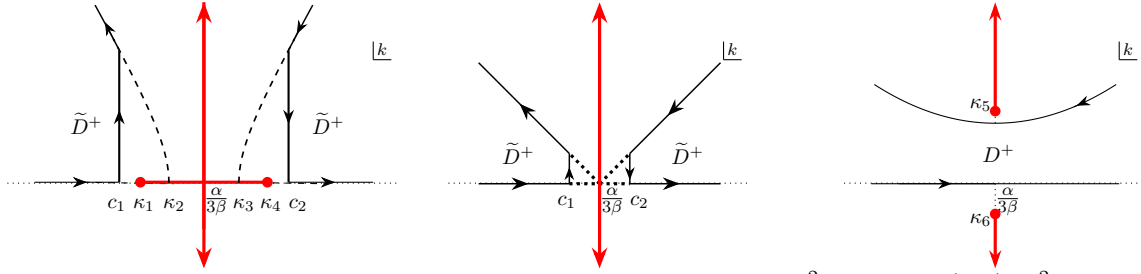
\begin{figure}[ht]
	\centering
	\vspace{2.5cm}
	\begin{tikzpicture}[scale=.8]
		\pgflowlevelsynccm
		\draw[line width=.5pt, black, dotted](-3.25,0.01)--(3.25,0.01);
		\draw[line width=1.5pt, red](0,-0.75)--(0, 3.00);
  \draw[Stealth-Stealth, line width=1.5pt, red]
(0,-1.4)--(0,3.00);
		\draw[line width=.5pt, black](2.8,2.35)--(2.8,2.05);
		\draw[line width=.5pt, black](2.8,2.05)--(3.1,2.05);
		\node[] at (2.89, 2.225) {\fontsize{8}{8} $k$};
        \node[] at (-1.5, -0.35) {\fontsize{10}{10} $c_1$};
        \node[] at (1.5, -0.35) {\fontsize{10}{10} $c_2$};
		\node[] at (-0.55, -0.35) {\fontsize{10}{10} $\kappa_2$};
		\node[] at (0.6, -0.35) {\fontsize{10}{10} $\kappa_3$};
		\node[] at (-1.05, -0.35) {\fontsize{10}{10} $\kappa_1$};
		\node[] at (1.05, -0.35) {\fontsize{10}{10} $\kappa_4$};
		\node[] at (-2, 1) {\fontsize{10}{10} $\tilde{D}^+$};
		\node[] at (2, 1) {\fontsize{10}{10} $\tilde{D}^+$};
		\node[] at (0.15, -0.3) {\fontsize{11}{11} $\frac{\alpha}{3\beta}$};
		\draw[middlearrow={Stealth[scale=1.3,reversed]}, black] (-1.6,2.6) -- (-58:-2.8);
		\draw[middlearrow={Stealth[scale=1.3]}, black] (1.6,2.6) -- (58:2.8);
		\draw[middlearrow={Stealth[scale=1.3, reversed]}, black, line width=.5pt] (-1.4,0.025) -- (-0.5:-2.8);
	\draw[line width=.5pt, black, dashed](-2,0.025)--(-0.5, 0);
		\draw[line width=.5pt, black, dashed](0.5,0.025)--(2, 0);
		\draw[middlearrow={Stealth[scale=1.3]}, black, line width=.5pt] (1.4,0.025) -- (0.5:2.8);
		\draw[domain=1.4:1.8, variable=\x, smooth, black, line width=0.7pt] plot ({\x}, {sqrt(3*\x*\x-1)});
		\draw[domain=-1.4:-1.8, variable=\x, smooth, black, line width=0.7pt] plot ({\x}, {sqrt(3*\x*\x-1)});
		\draw[line width=1.5pt, red](-1.05,0.025)--(1.05,0.025);
        \fill[red] (-1.05,0.025) circle (2.5pt);
        \fill[red] (1.05,0.025) circle (2.5pt);
		\draw[middlearrow={Stealth[scale=0.9]}, line width=1pt, black](-1.4,0.025)--(-1.4, 2.2);
			\draw[middlearrow={Stealth[scale=0.9, reversed]}, line width=1pt, black](1.4,0.025)--(1.4, 2.2);
				\draw[
		domain=-0.5774:-1.5,
		samples=150,
		variable=\x,
		black,
		line width=0.8pt,
		dashed,                 
		line cap=butt           
		]
		plot ({\x}, {sqrt(3*\x*\x-1)});
		
		\draw[
domain=0.5774:1.5,
samples=150,
variable=\x,
black,
line width=0.8pt,
dashed,                 
line cap=butt           
]
plot ({\x}, {sqrt(3*\x*\x-1)});

	\end{tikzpicture}
	\hspace*{5cm}
	\begin{tikzpicture}[scale=0.8]
		\pgflowlevelsynccm

		\draw[line width=.5pt, black, dotted](-2.5,0.01)--(2.5,0.01);
		\draw[line width=.5pt, black, dotted](0,-0.7)--(0, 2.35);
		\draw[middlearrow={Stealth[scale=1.3]}, line width=.7pt, black](-2,0)--(-0.5,0);
        \node[] at (-0.6,-0.3) {\fontsize{10}{10} $c_1$};
        \node[] at (0.6, -0.3) {\fontsize{10}{10} $c_2$};
		\draw[middlearrow={Stealth[scale=1.3]}, line width=.7pt, black](0.5,0)--(2,0);
		\draw[middlearrow={Stealth[scale=1]}, line width=.7pt, black](-0.5,0)--(-0.5,0.5);
		\draw[middlearrow={Stealth[scale=1.3]}, line width=.7pt, black](-0.5,0.5)--(-2,2);
		\draw[middlearrow={Stealth[scale=1,reversed]}, line width=.7pt, black](0.5,0)--(0.5,0.5);
		\draw[middlearrow={Stealth[scale=1.3,reversed]}, line width=.7pt, black](0.5,0.5)--(2,2);
		\draw[line width=1.5pt, black, dotted](0,0)--(-0.5, 0.5);
		\draw[line width=1.5pt, black, dotted](0,0)--(0.5, 0.5);
		\draw[line width=1.5pt, black, dotted](-0.5,0)--(0.5, 0);
          \draw[Stealth-Stealth, line width=1.5pt, red]
(0,-1.4)--(0,3.00);
		\draw[line width=.5pt, black](2.05,2.35)--(2.05,2.05);
		\draw[line width=.5pt, black](2.05,2.05)--(2.35,2.05);
		\node[] at (2.14, 2.225) {\fontsize{8}{8} $k$};
		\node[] at (0.15, -0.3) {\fontsize{11}{11} $\frac{\alpha}{3\beta}$};
		\node[] at (-1.4, 0.5) {\fontsize{10}{10} $\tilde{D}^+$};
		\node[] at (1.4, 0.5) {\fontsize{10}{10} $\tilde{D}^+$};
			\end{tikzpicture}
				\hspace*{5cm}
	\begin{tikzpicture}[scale=.8]
		\pgflowlevelsynccm
		\draw[line width=.5pt, black, dotted](-2.5,0.01)--(2.5,0.01);
		\draw[line width=.5pt, black, dotted](0,1)--(0, 2.35);
		\draw[line width=.5pt, black, dotted](0,0)--(0, -0.7);
        
		\draw[line width=.5pt, black](-2.05,0)--(2.05,0);
		\draw[line width=.5pt, black](2.05,2.35)--(2.05,2.05);
		\draw[line width=.5pt, black](2.05,2.05)--(2.35,2.05);
		\node[] at (2.14, 2.225) {\fontsize{8}{8} $k$};
		\node[] at (0.15, -0.3) {\fontsize{11}{11} $\frac{\alpha}{3\beta}$};
		\node[] at (0, 0.5) {\fontsize{10}{10} ${D}^+$};
		\draw[middlearrow={Stealth[scale=1.3]}, black] (-1.6,0) -- (1.6:0);
		\draw[domain=-2.0:2, variable=\x, smooth, black, line width=0.5pt] plot ({\x}, {(0.4*(\x))^2+1});
		\draw[middlearrow={Stealth[scale=1.3]}, black] (1.4,1.3136) -- (1.3,1.2704);
\draw[-{Stealth[length=3mm, width=2mm]}, line width=1.5pt, red]
(0,1.2) -- (0,3.0);
\fill[red] (0,1.2) circle (2.5pt);
\draw[-{Stealth[length=3mm, width=2mm]}, line width=1.5pt, red]
(0,-0.5) -- (0,-1.4);
\fill[red] (0,-0.5)  circle (2.5pt);
\node[] at (-0.3, 1.3) {\fontsize{10}{10} $\kappa_5$};
\node[] at (-0.3, -0.6) {\fontsize{10}{10} $\kappa_6$};
	\end{tikzpicture}
	\vspace{7mm}
	\caption{Red lines or rays show the relevant branch cuts for $\alpha^2+3\beta\delta>0$ (left),  $\alpha^2+3\beta\delta=0$ (middle), and $\alpha^2+3\beta\delta<0$ (right).  In the left and middle figures, $\tilde{D}^+$ is an open region obtained from ${D^+}$ by removing a bounded portion of it that intersects the branch cut. Specific values of $c_1,c_2$ will be chosen later.}
	\label{figure2}
\end{figure}

We introduce the notation
\begin{equation}\label{Gamma}
	\Gamma:=\left\{	\begin{array}{ll}
		\partial \tilde{D}^+, &\alpha^2+3\beta\delta\geq0, 
		\\
		\partial D^+, &\alpha^2+3\beta\delta<0.
	\end{array}
	\right.
\end{equation} Now, using {{simple complex analytical arguments}}, we can rewrite \eqref{intrep_D+} as follows:
\eqs{
	\begin{split}
		u(x, t) &= \frac{1}{2 \pi} \int_{-\infty}^{\infty} e^{i k x - \omega t} \b{\hat u_0 (k) - i \int_0^t e^{\omega t'} \hat f(k, t') dt'} dk\\
		&\quad + \frac{1}{2 \pi} \int_{\Gamma} e^{i k x -\omega t} \Big[\beta \tilde g_2 (\omega, t) + i \p{\beta k - \alpha} \tilde g_1 (\omega, t) - \p{\beta k^2 - \alpha k - \delta} \tilde g_0 (\omega, t)\Big] dk.
	\end{split} \alignlabel{intrep_gamma}}

In order to eliminate the unknown $t$-transform of the second order boundary trace from the above formula, we will benefit from the global relation \eqref{globalrelation} in view of the lemma below.

\begin{lemma}\label{validity}
	If $k\in{D^+}$ ($k\in \overline{D^+}$), then Im $\nu_+(k)> 0$ (Im $\nu_+(k)\ge 0$), and Im $\nu_-(k)< 0$ (Im $\nu_-(k)\le 0$).
\end{lemma}

\begin{proof} 
Set the notation $\displaystyle r(k)=\big( k-\frac{\alpha}{3\beta} \big)^2-\frac{4(\alpha^2+3\beta\delta)}{9\beta^2}$, $a(k)=\text{Re }r(k)$, $b(k)=\text{Im }r(k)$, $A(k)=\text{Re }{r(k)^\frac12}$, $B(k)=\text{Im }{r(k)^\frac12}$, $x=k_R-\frac{\alpha}{3\beta}$, and $y=k_I$.  Note that $Arg(r(k)^{\frac12})\in (-\frac{\pi}{2},\frac{\pi}{2}]$.  Therefore, $A(k)\ge 0$ and $\text{Im }\nu_-(k)=-\frac12 k_I-\frac{\sqrt{3}}{2}A(k)<0$ ($\le 0$) if $k\in D^+$ (if $k\in \overline{D^+}$). 

It is easy to show $\displaystyle A(k)=\sqrt{\frac{|r(k)|+a(k)}{2}}$. We claim that for $k\in D^+$, $\displaystyle A(k)>\frac{y}{\sqrt{3}}$, or equivalently $\displaystyle |r(k)|>\frac{2y^2}{3}-a(k)$.  We can assume without loss of generality the case where the right hand side of last claimed inequality is non-negative (otherwise it readily holds).  Then, \begin{equation*} |r(k)|>\frac{2y^2}{3}-a(k)\iff a^2(k)+b^2(k)>\left(\frac{2y^2}{3}-a(k)\right)^2\iff b^2(k)>\frac{4y^4}{9}-\frac{4y^2a(k)}{3}\end{equation*}  
\begin{equation*} \overbrace{\iff}^{b(k)=2xy,\, \text{divide by }4y^2} x^2>\frac{y^2}{9}-\frac{a(k)}{3} \overbrace{\iff}^{a(k) = x^2-y^2-\frac{4(\alpha^2+3\beta\delta)}{9\beta^2}} 3x^2 > y^2 +  \frac{\alpha^2+3\beta\delta}{3\beta^2} \overbrace{\iff}^{\eqref{Dplus}} k\in D^+.\end{equation*}  
Hence, our claim holds and $\text{Im }\nu_+=-\frac{1}{2}y+\frac{\sqrt{3}}{2}A(k)>0$ for $k\in D^+$.  Note that $\text{Im }\nu_+$ can vanish for some $k\in \overline{D^+}$, though it cannot be negative.
\end{proof}

\begin{remark}
    From Lemma \ref{validity}, it follows that at most one unknown boundary trace can be eliminated among the three boundary traces emerging upon integration by parts. This observation justifies why two boundary conditions need to be given for well-posedness when $\beta<0$ (contrasting with the single boundary condition case $\beta>0$~\cite{amo2024} where only one boundary condition must be given).
\end{remark}

Using Lemma \ref{validity} and the global relation \eqref{globalrelation} in view of $\omega(k)=\omega(\nu_+(k))$, we deduce for $k\in \overline{D^+}$ that
\begin{equation}\label{globalrelation_new}
	e^{\omega t}\hat{u}(\nu_+,t)=\hat{u}_0(\nu_+)- i \int_0^t e^{\omega t'} \hat f(\nu_+, t') dt'+(-\beta \nu_+^2+\alpha \nu_++\delta)\tilde{g}_0(\omega,t)+(i\beta \nu_+-i\alpha)\tilde{g}_1(\omega,t)+\beta\tilde{g}_2(\omega,t).
\end{equation}
Solving \eqref{globalrelation_new} for $\tilde{g}_2(\omega,t)$ and using this in \eqref{intrep_gamma}, we obtain 

\eqs{
	\begin{split}
		u(x, t) &= \frac{1}{2 \pi} \int_{-\infty}^{\infty} e^{i k x - \omega t} \b{\hat u_0 (k) - i \int_0^t e^{\omega t'} \hat f(k, t') dt'} dk\\
        &\quad - \frac{1}{2 \pi} \int_{\Gamma} e^{i k x -\omega t}\b{ \hat{u}_0(\nu_+)dk - i\int_0^t e^{\omega t'} \hat f(\nu_+, t') dt'} dk\\
        &\quad + \frac{1}{2 \pi} \int_{\Gamma} e^{i k x -\omega t} \Big[(k-\nu_+)(\alpha-\beta k-\beta \nu_+)\tilde{g}_0[\omega,t]+i\beta(k-\nu_+)\tilde{g}_1[\omega,t]\Big] dk\\
        &\quad -\frac{1}{2 \pi} \int_{\Gamma}e^{i k x}\hat{u}(\nu_+,t)dk.
	\end{split} \alignlabel{intrep_gamma_new}}
{By standard arguments of the Fokas's method}, the last integral involving $\hat{u}(\nu_+,t)$ vanishes. Moreover, in the integrals over $\Gamma$, the terms $\int_0^t e^{\omega t'} \hat f(\nu_+, t') dt'$, $\tilde{g}_0(\omega,t)$, and $\tilde{g}_1(\omega,t)$ can be replaced with 
$\int_0^T e^{\omega t'} \hat f(\nu_+, t') dt'$, $\tilde{g}_0(\omega,T)$, and $\tilde{g}_1(\omega,T)$, respectively.  Therefore, the explicit solution formula for the full linear ibvp takes the form
\eqs{
	\begin{split}
		u(x, t) &= \frac{1}{2 \pi} \int_{-\infty}^{\infty} e^{i k x - \omega t} \b{\hat u_0 (k) - i \int_0^t e^{\omega t'} \hat f(k, t') dt'} dk\\
        &\quad - \frac{1}{2 \pi} \int_{\Gamma} e^{i k x -\omega t}\b{ \hat{u}_0(\nu_+)dk - i\int_0^T e^{\omega t'} \hat f(\nu_+, t') dt'} dk\\
        &\quad + \frac{1}{2 \pi} \int_{\Gamma} e^{i k x -\omega t} \Big[(k-\nu_+)(\alpha-\beta k-\beta \nu_+)\tilde{g}_0(\omega,T)+i\beta(k-\nu_+)\tilde{g}_1(\omega,T)\Big] dk, \quad t\in (0,T).
	\end{split} \alignlabel{intrep_gamma_new}}

		\section{Sobolev estimates on the half-line} \label{red-s}
		
		We begin with a \textit{reduced} initial-boundary value problem:
		\eq{\label{HNLSreduced}
			\begin{aligned}
				&i v_t + i \beta v_{xxx} + \alpha v_{xx} + i \delta v_x = 0, \quad x>0, \ t>0,
				\\
				&v(x, 0) = 0,\\
				&v(0, t) = \psi_0(t), \quad v_x (0, t) = \psi_1(t),
			\end{aligned}
		}
		where the boundary data $\psi_0, \psi_1$ are globally defined on $\R$ but only supported in a compact set $[0, T']$, i.e.
		\eq{\label{supp-cond}
			\text{supp}(\psi_0) \subset [0, T'], \quad \text{supp}(\psi_1) \subset [0, T'].
		}
	 Notably, in addition to the Dirichlet condition, we now also have a Neumann condition compared to the case of $\beta>0$~\cite{amo2024}. Applying \eqref{intrep_gamma_new} to the reduced ibvp and using $\mathcal{F}[\psi_j](i\omega)=\tilde{\psi}_0(\omega,T)$, $j=0,1$, we get the following representation formula for $v$:
\eqs{
	\begin{split}
		v(x, t) &=\sum_{j=0}^1\frac{1}{2 \pi} \int_{\Gamma} e^{i k x -\omega(k) t} \phi_j(k)\mathcal{F}[\psi_j](i\omega(k)) dk, \quad t\in (0,T),
	\end{split} \alignlabel{sol_for_reduced}}
    where $$\phi_0(k)=(k-\nu_+(k))(\alpha-\beta k-\beta \nu_+(k))\quad \text{ and }\quad \phi_1(k)=i\beta(k-\nu_+(k)).$$
    
\subsection*{Real parametrization of $\Gamma$}
We find a parameterization of $\Gamma$ with respect to a real variable, depending on the sign of $\alpha^2+3\beta\delta$:
\begin{eqnarray}\label{paramet1}
	\Gamma=\left\{\begin{array}{l}
		\gamma_1\cup\gamma_2\cup\gamma_3\cup(-\gamma_4)\cup(-\gamma_5)\cup\gamma_6,\qquad\text{if }\alpha^2+3\beta\delta\geq0,\\
		\gamma_1\cup\gamma_3\cup(-\gamma_4)\cup\gamma_6,\qquad\text{if }\alpha^2+3\beta\delta<0.
	\end{array}	\right.
\end{eqnarray}
where
\begin{equation}\label{paramet2}
	\begin{aligned}
		&\gamma_1(m)=m,\qquad -\infty<m\leq c_1;
		\\
		&\gamma_2(m)=c_1+im,\qquad 0< m<\lambda;
		\\
		&\gamma_3(m)=\frac{\alpha+\sqrt{3\beta^2m^2+\alpha^2+3\beta\delta}}{3\beta}+im,\qquad \lambda\leq m<\infty;
		\\
		&\gamma_4(m)=\frac{\alpha-\sqrt{3\beta^2m^2+\alpha^2+3\beta\delta}}{3\beta}+im,\qquad \lambda\leq m<\infty;
		\\
		&\gamma_5(m)=c_2+im,\qquad 0< m<\lambda;
		\\
		&\gamma_6(m)=m,\qquad c_2\leq m<\infty
	\end{aligned}
\end{equation}
such that $c_1=\frac{\alpha+\sqrt{3\beta^2\lambda^2+\alpha^2+3\beta\delta}}{3\beta}$, $c_2=\frac{\alpha-\sqrt{3\beta^2\lambda^2+\alpha^2+3\beta\delta}}{3\beta}$, and the fixed $\lambda$ is chosen so that
\begin{eqnarray}\label{cursor}
	\left\{\begin{array}{l}
		\lambda=\frac{\sqrt{-3(\alpha^2+3\beta\delta)}}{-3\beta}\qquad\text{if }\alpha^2+3\beta\delta<0,\\
		\frac{\sqrt{\alpha^2+3\beta\delta}}{-\beta}<\lambda<\infty\qquad\text{if }\alpha^2+3\beta\delta\geq0.
	\end{array}	\right.
\end{eqnarray}
Notice that $c_1=c_2=\frac{\alpha}{3\beta}$ when $\alpha^2+3\beta\delta<0$.  Using the above parameterization, we can rewrite the solution in the form
\begin{equation}\label{qnr1}
	v(x,t):=\sum_{n=1}^{6}\sum_{b=0}^{1}v_{n,b}(x,t),
\end{equation}
where
\begin{equation}\label{qnr2}
	v_{n,b}(x,t)=\frac{1}{2\pi}\int_{I_n}e^{i\gamma_n(m)x-w(\gamma_n(m))t}\phi_b(\gamma_n(m))\mathcal{F}[\psi_b](i\omega(\gamma_n(m)))\gamma'_n(m)dm,
\end{equation}
with the exception of setting
\begin{equation}\label{qnr3}
	v_{2,0}=v_{2,1}=v_{5,0}=v_{5,1}\equiv0,\qquad\text{if }\alpha^2+3\beta\delta<0.
\end{equation}
In \eqref{qnr2}, the intervals, $I_n$, $n=1,...,6$, are defined by
\begin{equation}\label{qnr4}
	\begin{aligned}
		&I_1=(-\infty,c_1];\quad I_2=(-I_5)=(0,\lambda);\quad I_3=(-I_4)=[\lambda,\infty);\quad I_6=[c_2,\infty),
	\end{aligned}
\end{equation}
where the minus sign in front of $I_4$ and $I_5$ in \eqref{qnr4} is not used in the algebraic sense; its purpose is to emphasize that the orientation of the integration is reversed with respect to the endpoints of the given interval.

We prove the following theorem.
\begin{theorem}[Sobolev estimates]
\label{reduced}
Let $s\ge 0$, and $\psi_0\in H_t^{\frac{s+1}{3}}(\mathbb{R})$, $\psi_1\in H_t^{\frac{s}{3}}(\mathbb{R})$ satisfy \eqref{supp-cond}.  Then, the solution of \eqref{HNLSreduced} represented by the formula \eqref{sol_for_reduced} satisfies 
\begin{equation}
    \|v(\cdot,t)\|_{H^s_x(\mathbb{R}_+)}\lesssim (1+e^{cT'}\sqrt{T'})\sum_{b=0}^1\|\psi_b\|_{H_t^{\frac{s+1-b}{3}}},\quad t\in [0,T'],
\end{equation} where $c$ and the constant of the inequality may depend only on the fixed parameters such as $s,\alpha,\beta$, and $\delta$.
    
\end{theorem}

\subsection*{Sobolev estimate for $v_{1,b}$ and $v_{6,b}$, $b=0,1$} The estimates of $v_{1,b}$ and $v_{6,b}$ are done in the same way. Therefore, we only give details for $v_{1,b}$. Notice that $v_{1,b}(x,t)=\mathcal{F}^{-1}[\tilde{V}_{1,b}(\cdot,t)](x)$ for $x\in \mathbb{R}_+$, where $\tilde{V}_{1,b}(m,t)=
	e^{-w(m)t}\phi_b(m)\mathcal{F}[\psi_b](i\omega(m))\chi_{m\le c_1}(m)$. Let $V_{1,b}(x,t)=\mathcal{F}^{-1}[\tilde{V}_{1,b}(\cdot,t)](x)$ for $x\in \mathbb{R}$. 
It follows that 
\begin{eqnarray*}
	||v_{1,b}(\cdot,t)||_{H^s_x(\mathbb{R}_+)}^2
	&\le ||V_{1,b}(\cdot,t)||_{H^s_x(\mathbb{R})}^2 &=\int_{-\infty}^{c_1}(1+m^2)^s|\phi_b(m)|^2|\mathcal{F}[{\psi}_b](iw(m))|^2dm.
\end{eqnarray*}
We change variables via $\tau=i\omega(m)=\beta m^3-\alpha m^2-\delta m=:p(m)$, $\text{dom}(p)= (-\infty,c_1]$. It should be noted that {$p$ is monotonically decreasing on its domain}, and hence is one-to-one and admits an inverse in its range. Therefore, we may write
$$||V_{1,b}(\cdot,t)||_{H^s_x(\mathbb{R})}^2=\int_{iw(c_1)}^\infty (1+m_\tau^2)^s|\phi_b(m_\tau)|^2|\mathcal{F}[{\psi}_b](\tau)|^2\frac{1}{|p'(m_\tau)|}d\tau,$$ where $m_\tau=p^{-1}(\tau)$, $\tau\in [i\omega(c_1),\infty)$.  {It is not difficult to show that there exists a positive constant  $c$ depending on $\alpha,\beta,\delta,s,$ and $b$ such that}
\begin{equation}\label{asymptotic}
	\sup_{\tau\geq iw(c_1)} \frac{(1+m_\tau^2)^s|\phi_b(m_\tau)|^2}{(1+\tau^2)^{\frac{s+1-b}{3}}|p'(m_\tau)|}\leq c.
\end{equation}

Thus, 
\begin{equation}\label{v1bpsib}||v_{1,b}(\cdot,t)||_{H^s_x(\mathbb{R}_+)}^2\le c\int_{i\omega(c_1)}^\infty(1+\tau^2)^{\frac{s+1-b}{3}}|\mathcal{F}[{\psi}_b](\tau)|^2d\tau\lesssim \|\psi_b\|_{H_t^{\frac{s+1-b}{3}}(\mathbb{R})}^2.\end{equation}

\subsection*{Sobolev estimate for $v_{2,b}$ and $v_{5,b}$, $b=0,1$} We give details for $v_{2,b}$ since $v_{5,b}$ can be treated in a similar way.  Also we consider only the case $\alpha^2+3\beta\delta\geq0$ since otherwise $v_{2,b}=v_{5,b}\equiv 0$. Differentiating \eqref{qnr2} $j$-times ($j\ge 0$) and taking the $L_x^2(\mathbb{R}_+)$-norm, we get
\begin{align*}
	\|\partial_x^j v_{2,b}(\cdot,t)\|_{L^2_x(\mathbb{R}_+)}^2 &=\frac{1}{4\pi^2}\int_0^\infty\left|\int_{0}^{\lambda} (i\gamma_2(m))^je^{i\gamma_2(m)x-w(\gamma_2(m))t}\phi_b(\gamma_2(m))\tilde{\psi}_b(w(\gamma_2(m)),T')\;dm\right|^2dx\\
    &\lesssim \int_0^\infty\left(\int_{0}^{\lambda} e^{-mx-\text{Re }(\omega(\gamma_2(m))t)}|\gamma_2(m)|^{j}|\phi_b(\gamma_2(m))|\;|\tilde{\psi}_b(w(\gamma_2(m)),T')|\;dm \right)^2  dx \\
    &\lesssim \int_0^\lambda\left(e^{-\text{Re }(\omega(\gamma_2(m))t)}|\gamma_2(m)|^{j}|\phi_b(\gamma_2(m))|\;|\tilde{\psi}_b(w(\gamma_2(m)),T')|\right)^2dm\\
     &\lesssim T'e^{cT'}\|\psi_b\|_{L^2_t(0,T')}^2,
\end{align*} where the constants of the inequalities may depend on $j, \alpha,\beta,\gamma$, and $\lambda$ above. The second inequality in above estimate uses {boundedness of the Laplace-transform from $L^2_m(\mathbb{R}_+)$ into $L^2_x(\mathbb{R}_+)$} and the last inequality uses the simple estimate $$|\tilde{\psi}_b(w(\gamma_2(m)),T')|\le \int_0^{T'}|\psi_b(t')|dt'\le (T')^{\frac12}\|\psi_b\|_{L^2_t(0,T')}$$ and the observation $\text{Re }(-w(\gamma_2(m))t)=-\beta tm(\lambda^2-m^2)\le \frac{-2\beta\lambda^3 T'}{3\sqrt{3}}$ for $0\le m\le \lambda$ and $0\le t\le T'$.
Summing the relevant estimates from $j=0$ to $j=\lceil s\rceil$ and interpolating we find for $s\ge 0$ that 
\begin{equation}\label{q20est}
	\|v_{2,b}(\cdot,t)\|_{H^s_x(\mathbb{R}_+)}\lesssim e^{cT'}\sqrt{T'}\|\psi_b\|_{L_t^2(0,T')}\le e^{cT'}\sqrt{T'} \|\psi_b\|_{H_t^{\frac{s+1-b}{3}}(0,T')}.
\end{equation}

\subsection*{Sobolev estimate for $v_{3,b}$ and $v_{4,b}$, $b=0,1$}
We will prove the Sobolev estimate for $v_{3,b}$; the estimate for $v_{4,b}$ can be proved similarly.  
Differentiating \eqref{qnr2} $j$-times ($j\ge 0$) and taking the $L_x^2(\mathbb{R}_+)$-norm, we get

\eqs{
	\begin{split}
	\|\partial_x^j v_{3,b}(\cdot,t)\|_{L^2_x(\mathbb{R}_+)}^2 & =\frac{1}{4\pi^2}\int_0^\infty\left|\int_{\lambda}^{\infty} (i\gamma_3(m))^je^{i\gamma_3(m)x-w(\gamma_2(m))t}\phi_b(\gamma_3(m))\mathcal{F}[\psi_b](iw(\gamma_3(m)))\;dm\right|^2dx\\
    &\lesssim \int_0^\infty\left(\int_{\lambda}^{\infty} e^{-mx}|\gamma_3(m)|^{j}|\phi_b(\gamma_3(m))|\;|\mathcal{F}[\psi_b](iw(\gamma_3(m)))|\;dm \right)^2  dx \\
    &\lesssim \int_\lambda^\infty\left(|\gamma_3(m)|^{j}|\phi_b(\gamma_3(m))|\;|\mathcal{F}[\psi_b](iw(\gamma_3(m)))|\right)^2dm.
\end{split}
}

We change variable by setting $\tau = i\omega(\gamma_3(m)) =: p(m)$, $m\in [\lambda,\infty)$.  We have $\tau\in \mathbb{R}$ as $\gamma_3(m)\in \partial D^+$ and $\text{Re } \omega(\gamma_3(m))=0$. Also, Range$(\tau)=(-\infty, i\omega(c_1+i\lambda)]$ and $p$ is monotone decreasing. Therefore,
\eqs{
	\begin{split}
	\|\partial_x^j v_{3,b}(\cdot,t)\|_{L^2_x(\mathbb{R}_+)}^2  
    &\lesssim \int_{-\infty}^{i\omega(c_1+i\lambda)}\left(|\gamma_3(m_\tau)|^{j}|\phi_b(\gamma_3(m_\tau))|\;|\mathcal{F}[\psi_b](\tau)|\right)^2\frac{1}{|p'(m_\tau)|}d\tau,
\end{split}
} where $m_\tau=p^{-1}(\tau)$.  {One can show that there exists a positive constant  $c$ depending on $\alpha,\beta,\delta,j,$ and $b$ such that}
\begin{equation}\label{asymptotic}
	\sup_{\tau\le iw(c_1+i\lambda)} \frac{|\gamma_3(m_\tau)|^{2j}\phi_b(\gamma_3(m_\tau))|^2}{(1+\tau^2)^{\frac{j+1-b}{3}}|p'(m_\tau)|}\leq c.
\end{equation}
Thus, the following estimate follows for $s\in \mathbb{N}_0$, but also via interpolation for all real $s\ge 0$:
$$||v_{3,b}(\cdot,t)||_{H^s_x(\mathbb{R}_+)}^2\lesssim\int_{-\infty}^{i\omega(c_1+i\lambda)}(1+\tau^2)^{\frac{s+1-b}{3}}|\mathcal{F}[{\psi}_b](\tau)|^2d\tau\lesssim \|\psi_b\|_{H_t^{\frac{s+1-b}{3}}(\mathbb{R})}^2.$$

\section{Sobolev and Strichartz estimates for Cauchy problems}\label{ivp-s}

In this section, we review the well-posedness features of homogeneous and non-homogeneous Cauchy problems for the higher-order linear Schrödinger equation.
\subsection{Homogeneous Cauchy problem}
We consider the initial value problem
 \begin{eqnarray}\label{cauchy1}
 	\begin{array}{l}
 		iy_t+i\beta y_{xxx}+\alpha y_{xx}+i\delta y_x = 0,  \qquad x,t\in\mathbb{R},\\
 		y(x,0) = y_0(x), \qquad x\in \mathbb{R},
 	\end{array}
 \end{eqnarray}	where $y_0\in H^s_x(\mathbb{R})$, $\beta<0$ (or more generally $\beta\neq 0$), $\alpha,\delta\in \mathbb{R}$.

\begin{theorem}[Sobolev estimates]\label{cauchylemma}
	Let $s\in\mathbb R$, $b\in \{0,1\}$. The unique solution of the Cauchy problem \eqref{cauchy1}, denoted by $y=S[y_0;0]$, belongs to $C(\mathbb{R}_t;H_x^s(\mathbb{R}))$ and satisfies the conservation law
	\begin{equation}\label{cuachyspaceest}
		\left\| y(\cdot ,t) \right\|_{H_x^s(\mathbb{R})}= \left\| y_0 \right\|_{H_x^s(\mathbb{R})},\quad t\in\mathbb{R}.
	\end{equation}
	Moreover, if  $\alpha^2+3\beta\delta\ge 0$, then $\partial_x^by\in C(\mathbb{R}_x;H_t^{\frac{s+1-b}{3}}(-T,T))$ for any $T>0$ and there exists a constant $c=c(s,T\alpha,\beta,\delta)\geq 0$ 
    such that
\begin{equation}\label{cauchyextra3}
		\sup_{x\in\mathbb{R}}\left\|\partial_x^by(x,\cdot)\right\|_{H_t^{\frac{s+1-b}{3}}(-T,T)}\leq 
        c \left\|y_0\right\|_{H_x^s(\mathbb{R})},
	\end{equation}
	while if $\alpha^2+3\beta\delta<0,$ then $\partial_x^by\in C(\mathbb{R}_x;H_t^{\frac{s+1-b}{3}}(\mathbb{R}))$ and there is a constant $c=c(s,\alpha,\beta,\delta)\geq 0$ such that
	\begin{equation}\label{cauchyextra3*}
		\sup_{x\in\mathbb{R}}\left\|\partial_x^by(x,\cdot)\right\|_{H_t^{\frac{s+1-b}{3}}(\mathbb{R})}\leq c \left\|y_0\right\|_{H_x^s(\mathbb{R})}.
	\end{equation}
\end{theorem}

\begin{proof}
    The proof in the case $b=0$ can be done by adapting the arguments in the proof of \cite[Theorem 3]{amo2024} to the case $\beta<0$. The case $b=1$ follows from differentiating \eqref{cauchy1} in $x$ and applying estimates of the case $b=0$ to $y_x$ while replacing $y_0$ with $y_0'$ and  $s$ with $s-1$.
\end{proof}

\begin{theorem}[Strichartz estimate]\label{homStr}
	Let $s\in\mathbb{R}$ and $(\mu,r)$ be higher-order Schr\"{o}dinger admissible, i.e. 	$\mu,r\geq 2,$ and $\displaystyle \frac{3}{\mu}+\frac{1}{r}=\frac12.$ Then, the solution of the homogeneous Cauchy problem \eqref{cauchy1} satisfies the Strichartz estimate
	\begin{equation}\label{stry}
		\left\|y\right\|_{L_{t}^{\mu}((0, T); H_{x}^{s,r}(\mathbb{R}))}\lesssim \left\|y_0\right\|_{H_{x}^{s}(\mathbb{R})},
	\end{equation}
    where $H^{s, r}(\mathbb{R})$ is the Bessel potential space with the norm 
\begin{equation}\label{bessel-def}
	\left\|f\right\|_{H^{s,r}(\mathbb{R})}
	:=
	\left\| \mathcal{F}^{-1}\left\{ \left(1+k^2\right)^{\frac{s}{2}} \mathcal F\{f\}(k)\right\}\right\|_{L^{r}(\mathbb{R})}.
\end{equation}
\end{theorem}
\begin{proof}
The proof of this theorem is identical to the proof of the Strichartz estimates in the case $\beta>0$ (see \cite[Theorem 4]{amo2024}).
\end{proof}

\subsection{Non-homogeneous Cauchy problem}
Consider the initial value problem 
\begin{eqnarray}\label{cauchy2}
	\begin{array}{l}
		iz_t+i\beta z_{xxx}+\alpha z_{xx}+i\delta z_x = F,  \quad (x,t)\in\mathbb{R} \times (0,T),\\
		z(x,0) = 0, \quad x\in \mathbb{R},
	\end{array}
\end{eqnarray}
where $\alpha, \delta\in\mathbb{R}$, $\beta<0$ (or more generally $\beta\neq0$), and $F\in L^2_t((0,T);H^s_x(\mathbb{R}))$.  Thanks to Duhamel's principle, the solution of the Cauchy problem \eqref{cauchy2}, denoted by $S[0;F]$, can be expressed as
\begin{equation}\label{duhamel}
	\begin{aligned}
		z(x,t)=S[0;F](x,t)&=-i\int_{0}^{t} S[F(\cdot,t');0](x,t-t')dt'
		\\
		&=-\frac{i}{2\pi}\int_{0}^{t}\int_{\mathbb{R}} e^{ikx-\omega(k)(t-t')}\widehat{F}(k,t') dkdt',
	\end{aligned}
\end{equation}
where for each $t'\in[0, t]$, $S[F(\cdot,t');0]$ denotes the solution to the homogeneous Cauchy problem \eqref{cauchy1} with initial data $F(x, t')$.  We have the following theorems whose proofs are identical to proofs of \cite[Theorem 5, Theorem 6]{amo2024} except for the case $b=1$ which can be established by differentiating \eqref{cauchy2} in $x$ and applying estimates of the case $b=0$ to $z_x$ while replacing $F$ with $F_x$ and  $s$ with $s-1$.
\begin{theorem}[Sobolev estimates]\label{Nonhomthm}Let $b\in \{0,1\}$. Then, the unique solution of \eqref{cauchy2} satisfies the space estimate
	\begin{equation}\label{nonhthm1}
		\sup_{t\in[0,T]}\left\| z(\cdot,t) \right\|_{H_x^{s}(\mathbb{R})}\leq \left\|F\right\|_{L_t^1((0,T); H_x^{s}(\mathbb{R}))}, \quad s\in \mathbb{R}.
	\end{equation}
	Moreover, if  $-1+b\leq s \leq 2+b$ with $s\neq \frac 12+b$, then
	the following time estimate holds
\begin{equation}\label{nonhthm2}
\sup_{x\in\mathbb{R}}\left\| \partial_x^bz(x,\cdot) \right\|_{H_t^{\frac{s+1-b}{3}}(0,T)}\leq c \, 
\|F\|_{L_t^2((0,T); H_x^{s}(\mathbb{R}))}
	\end{equation}
	for some constant
    $c = c(s, T, \alpha, \beta, \delta) \geq 0$.
\end{theorem}
\begin{theorem}[Strichartz estimate]\label{NonHomStrThm}
	Let $s\in\mathbb{R}$ and $(\mu,r)$ be higher-order Schrödinger admissible. Then, the solution of the nonhomogeneous linear Cauchy problem \eqref{cauchy2} satisfies the Strichartz estimate
	\begin{equation}
		\left\|z\right\|_{L^\mu_t((0, T); H^{s,r}_x(\mathbb{R}))}\lesssim \|F\|_{L_t^1((0,T); H^s_x(\mathbb{R}))}.
	\end{equation}
\end{theorem} 

\section{Strichartz estimates on the half-line}\label{str-s}

We again consider the initial-boundary value problem \eqref{HNLSreduced} whose solution is given by the formula \eqref{sol_for_reduced}.  We establish the following theorem.

\begin{theorem}[Strichartz estimates]
\label{reduced-s}
Let $s\ge 0$, $(\mu,r)$ be higher-order Schrödinger admissible, and $\psi_0\in H_t^{\frac{s+1}{3}}(\mathbb{R})$, $\psi_1\in H_t^{\frac{s}{3}}(\mathbb{R})$ satisfy \eqref{supp-cond}.  Then, the solution of \eqref{HNLSreduced} represented by the formula \eqref{sol_for_reduced} satisfies 
\begin{equation}
    \left\|v\right\|_{L^\mu_t((0, T); H^{s,r}_x(\mathbb{R}_+))}\lesssim (1+e^{cT'}(T')^{\frac12+\frac 1 \mu})\sum_{b=0}^1\|\psi_b\|_{H_t^{\frac{s+1-b}{3}}},
\end{equation} 
where $c$ and the constant of the inequality may depend only on the fixed parameters such as $s,\alpha,\beta$, and $\delta$.
    
\end{theorem}

\subsection*{Reparametrization of $\Gamma$}
For the proof of Strichartz estimates, we first switch to a more convenient parametrization of $\Gamma$. To this end, we replace $\gamma_3$ and $\gamma_4$ in \eqref{paramet1}-\eqref{paramet2} with $\tilde{\gamma}_3$ and $\tilde{\gamma}_4$, respectively, where
\begin{equation}\label{paramet2_new}
	\begin{aligned}
		&\tilde{\gamma}_3(m)=m+i\sqrt{3\left(m-\frac{\alpha}{3\beta}\right)^2-\frac{\alpha^2+3\beta\delta}{3\beta^2}},\qquad -\infty< m\le c_1;
		\\
		&\tilde{\gamma}_4(m)=m+i\sqrt{3\left(m-\frac{\alpha}{3\beta}\right)^2-\frac{\alpha^2+3\beta\delta}{3\beta^2}},\qquad c_2\leq m<\infty.
	\end{aligned}
\end{equation}
Also, the intervals $I_3$ and $I_4$ are respectively replaced with $\tilde{I}_3$ and $\tilde{I}_4$, where $(-\tilde{I}_3)=(-\infty,c_1]$ and $(-\tilde{I_4})=[c_2,\infty)$ in \eqref{qnr2} and \eqref{qnr4}.
\subsection*{Strichartz estimate for $v_{1,b}$ and $v_{6,b}$, $b=0,1$}
We provide details only for $v_{1,b}$ because $v_{6,b}$ can be handled similarly.  Using \eqref{qnr2}, we have 
\begin{equation}\label{v1b}
	v_{1,b}(x,t)=\frac{1}{2\pi}\int_{-\infty}^{c_1}e^{imx-w(m)t}\phi_b(m)\mathcal{F}[\psi_b](i\omega(m))dm,\quad x\in \mathbb{R}_+.
\end{equation}
The above formula extends to all $x\in \mathbb{R}$ and we have $v_{1,b}=V_{1,b}|_{\mathbb{R}_+}$, where  $$V_{1,b}(x,t):= \mathcal{F}^{-1}[e^{-w(m)t}\phi_b(m)\mathcal{F}[\psi_b](i\omega(m))\chi_{m\le c_1}](x).$$ We notice that $V_{1,b}$ solves the initial-value problem below:
 \begin{eqnarray}\label{V1bcauchy1}
 	\begin{array}{l}
 		i\partial_tV_{1,b}+i\beta \partial_x^3V_{1,b}+\alpha\partial_x^2 V_{1,b}+i\delta\partial_xV_{1,b} = 0,  \qquad x,t\in\mathbb{R},\\
 		V_{1,b}(x,0) = V_{1,b}^0(x), \qquad x\in \mathbb{R},
 	\end{array}
 \end{eqnarray}	where $V_{1,b}^0(x)=\mathcal{F}^{-1}[\phi_b(m)\mathcal{F}[\psi_b](i\omega(m))\chi_{m\le c_1}](x)$. Therefore,  by Theorem \ref{homStr}, 
 \begin{equation}\label{V1bstry}
		\left\|v_{1,b}\right\|_{L_{t}^{\mu}((0, T); H_{x}^{s,r}(\mathbb{R}_+))}\le\left\|V_{1,b}\right\|_{L_{t}^{\mu}((0, T); H_{x}^{s,r}(\mathbb{R}))}\lesssim \left\|V_{1,b}^0\right\|_{H_{x}^{s}(\mathbb{R})} \lesssim \|\psi_b\|_{H_t^{\frac{s+1-b}{3}}(\mathbb{R})}
	\end{equation} for admissible $(\mu,r)$, where the last inequality is due to an argument similar to the one used in \eqref{v1bpsib}.
    
\subsection*{Strichartz estimate for $v_{2,b}$ and $v_{5,b}$, $b=0,1$}

Here, we consider $v_{2,b}$ ($v_{5,b}$ can be treated analogously). Also we only consider the case   $\alpha^2+3\beta\delta\ge 0$ since otherwise $v_{2,b}=0$.  In \eqref{q20est}, we have already proved the following for any $s'\ge 0$:
\begin{equation}
	\|v_{2,b}(\cdot,t)\|_{H^{s'}_x(\mathbb{R}_+)}\lesssim e^{cT'}\sqrt{T'}\|\psi_b\|_{L_t^2(0,T')},
\end{equation} where the constant of the estimate depends on $s'$.  We choose $s'$ large enough that $H^{s'}(\mathbb{R}_+)\hookrightarrow H^{s,r}(\mathbb{R_+})$, e.g., $s'=s+\frac12-\frac1r$. This follows from 
\[
H^{s'}(\mathbb{R}_+)
\xrightarrow{\,E\,}
H^{s'}(\mathbb{R})
\hookrightarrow
H^s_r(\mathbb{R})
\xrightarrow{\,R\,}
H^s_r(\mathbb{R}_+),
\] where $E$ and $R$ are the bounded extension and restriction operators respectively, and the continuous embedding in the middle is well known, see e.g. \cite[Theorem 2.7 (ii)]{Tr83}. Thus, we have 
\begin{equation}
\begin{split}
\|v_{2,b}(\cdot,t)\|_{L^\mu(0,T';H^{s,r}_x(\mathbb{R}_+))}&=\left(\int_0^{T'}\|v_{2,b}(\cdot,t)\|_{H^{s,r}_x(\mathbb{R}_+)}^\mu dt\right)^{1/\mu}\\ &\lesssim 
	\left(\int_0^{T'}\|v_{2,b}(\cdot,t)\|_{H^{s'}_x(\mathbb{R}_+)}^\mu dt\right)^{1/\mu}\lesssim e^{cT'}(T')^{\frac12+\frac1\mu}\|\psi_b\|_{L_t^2(0,T')}.
    \end{split}
\end{equation} 

\subsection*{Strichartz estimate for $v_{3,b}$ and $v_{4,b}$, $b=0,1$} 
Finally, we consider $v_{3,b}$ ($v_{4,b}$ is treated in a similar way). To this end, we rewrite $v_{3,b}$ in \eqref{qnr2} with respect to the parametrization \eqref{paramet2_new} as follows:
\begin{equation}\label{v3b}
	v_{3,b}(x,t)=-\frac{1}{2\pi}\int_{-\infty}^{c_1} e^{i\tilde{\gamma}_3(m)x-w(\tilde{\gamma}_3(m))t}\phi_b(\tilde{\gamma}_3(m))\mathcal{F}[\psi_b](i\omega(\tilde{\gamma}_3(m)))\tilde{\gamma}_3'(m)dm.
\end{equation}
Setting $\Psi_b(m):=-\phi_b(\tilde{\gamma}_3(m))\mathcal{F}[\psi_b](i\omega(\tilde{\gamma}_3(m)))\tilde{\gamma}_3'(m)\chi_{m\le c_1}$, we can rewrite $v_{3,b}$ in the form
\begin{equation}\label{v3b}
	v_{3,b}(x,t)=\frac{1}{2\pi}\int_{-\infty}^{\infty}\mathcal{K(y;x,t)}\Psi(y)dy,
\end{equation} where $\Psi(y)=\mathcal{F}^{-1}[\Psi_b](y)$ and $\mathcal{K}(y;x,t)=\int_{-\infty}^{c_1}e^{i\phi(m;x,y,t)}p(m;x)dm$ with $\phi(m;x,y,t)=m(x-y)+i\omega(\tilde{\gamma}_3(m))t$ and $p(m;x)=e^{-x\sqrt{3\left(m-\frac{\alpha}{3\beta}\right)^2-\frac{\alpha^2+3\beta\delta}{3\beta^2}}}.$ Using exactly the same strategy of \cite[Theorem 8]{amo2024}, one finds that for $s\ge 0$, $$\|v_{3,b}\|_{L^\mu(0,T';H_x^{s,r}(\mathbb{R_+}))}\lesssim \|\Psi\|_{H^s(\mathbb{R})}\lesssim \|\psi_b\|_{H_t^{\frac{s+1-b}{3}}(\mathbb{R})}.$$

\section{Local well-posedness of the nonlinear problem}\label{Sec_contraction}

\subsection{High-regularity setting - Proof of Theorem \ref{HighRegThem}}
Let $s\in (\frac12,2]$, $s\neq \frac32$, and set $X_T^s=C([0,T];H^s_x(\mathbb{R}_+))$.  We define the (solution) map \begin{equation}\label{PhiDef}
    \Phi(u):=y|_{Q_T}+z^u|_{Q_T}+v^{u}|_{(0,T)}, \quad u\in X_T^s, \quad Q_T=\mathbb{R}_+\times (0,T),
\end{equation}
where $y$ solves \eqref{cauchy1} with initial datum $y_0=\tilde{u}_0$, $\tilde{u}_0$ being a spatial extension of $u_0$ with $\|\tilde{u}_0\|_{H^s_x(\mathbb{R})}\lesssim\|u_0\|_{H^s_x(\mathbb{R}_+)}$ and given $u\in X_T^s$, $z^u$ solves \eqref{cauchy2} with $F=\kappa |\tilde{u}|^{\lambda-1}{\tilde{u}}$,  $\tilde{u}$ being a spatial extension of $u$ with $\|\tilde{u}(t)\|_{H^s_x(\mathbb{R})}\lesssim \|{u}(t)\|_{H^s_x(\mathbb{R}_+)}$ for each $t\in (0,T)$, and $v^u$ solves \eqref{HNLSreduced} for some $T'>T$ with $\psi_0$ and $\psi_1$ being temporal extensions of $g_0-y|_{x=0}-z^u|_{x=0}$ and $g_1-y_x|_{x=0}-z^u_x|_{x=0}$ supported on $[0,T')$, respectively through a continuous extension so that $\|\psi_0\|_{H_t^{\frac{s+1}{3}}(\mathbb{R})}\lesssim \|g_0-y|_{x=0}-z^u|_{x=0}\|_{H_t^{\frac{s+1}{3}}(0,T)}$ and $\|\psi_1\|_{H_t^{\frac{s}{3}}(\mathbb{R})}\lesssim \|g_1-y_x|_{x=0}-z^u_x|_{x=0}\|_{H_t^{\frac{s}{3}}(0,T)}$.  Here, we assume $u_0(0)=g_0(0)$ and in addition if $s\in (3/2,2]$, we further assume $u_0'(0)=g_1(0)$. The aforementioned extensions $\psi_0,\psi_1$ exist thanks to these compatibility conditions.  

Using Theorem \ref{cauchylemma} and due to the choice of $y_0$, we have
$
    \|y|_{Q_T}\|_{X_T^s} \lesssim \|u_0\|_{H^s_x(\mathbb{R}_+)}.
$
Next, Theorem \ref{Nonhomthm}, the choice of $\tilde{u}$ and the algebra property of $H^s_x(\mathbb{R})$ yield
$
    \|z^u|_{Q_T}\|_{X_T^s} \lesssim T \|u\|_{X_T^s}^{\lambda}$.  Moreover, for $T'=T+1$ (which ensures that the constants involved in the various linear estimates remain bounded as $T\to 0^+$), using {Theorem \ref{reduced} along with} the trace estimates of Theorem \ref{cauchy1} and Theorem \ref{Nonhomthm} where, in the latter case, $T$ is replaced by $T'$ and $F$ is replaced by $\widetilde F$ such that $\widetilde F = F$ on $(0, T)$ and $\widetilde F=0$ on $(T, T'$), 
    we find via the algebra property of $H^s_x(\mathbb{R})$ that
\begin{equation}
\begin{split}
    \|v^u|_{(0,T)}\|_{X_T^s}&\le  \|v^u\|_{X_{T'}^s}\leq c_T \left(\|\psi_0\|_{H_t^{\frac{s+1}{3}}(\mathbb{R})}+\|\psi_1\|_{H_t^{\frac{s}{3}}(\mathbb{R})}\right)\\
    &\leq c_T \left(\|g_0\|_{H_t^{\frac{s+1}{3}}(0,T)}+\|g_1\|_{H_t^{\frac{s}{3}}(0,T)}+\|u_0\|_{H^s_x(\mathbb{R}_+)}+ \sqrt{T} 
    \|u\|_{X_T^s}^\lambda\right).
    \end{split}
\end{equation}
Importantly, the constant $c_T$ remains bounded as $T\to0^+$.
Combining the estimates above, we obtain
\begin{equation}\label{Phiinvariance}
    \|\Phi(u)\|_{X_T^s}\le c_1(T)\|u_0\|_{H^s_x(\mathbb{R}_+)} +c_2(T)\left(\|g_0\|_{H_t^{\frac{s+1}{3}}(0,T)}+\|g_1\|_{H_t^{\frac{s}{3}}(0,T)}\right) +c_3(T)\|u\|_{X_T^s}^\lambda,
\end{equation} for some nonnegative constants $c_i(T)$, $i=1,2,3$ that depend on $T$ and fixed parameters of the problem such as $\alpha,\beta,\delta,\kappa, s$, where $c_1(T)$ and $c_2(T)$ remain bounded and $c_3(T)\rightarrow 0^+$ as $T\rightarrow 0^+$. Assume without loss of generality that either $u_0$ or else at least one of the $g_i$'s is nonzero (otherwise the zero solution satisfies the nonlinear ibvp). Set \begin{equation}\label{defrT}R_T=2\left(c_1(T)\|u_0\|_{H^s_x(\mathbb{R}_+)} +c_2(T)\left(\|g_0\|_{H_t^{\frac{s+1}{3}}(0,T)}+\|g_1\|_{H_t^{\frac{s}{3}}(0,T)}\right)\right).\end{equation}  Then, for any $u\in \overline{B_{R_T}(0)}\subset X_T^s$, \begin{equation}\label{c3RT}
     \|\Phi(u)\|_{X_T^s} \le \frac{1}{2}R_T +c_3(T)R_T^\lambda\le R_T
\end{equation} provided $T$ is small enough. That is, $\Phi$ is invariant on $\overline{B_{R_T}(0)}$. Using \cite[Lemma 5]{BO16} and linear estimates, for the differences, we have
\begin{equation}\label{clambdaTRT}
     \|\Phi(u_1)-\Phi(u_2)\|_{X_T^s} \le c(T,\lambda)R_T^{\lambda-1}\|u_1-u_2\|_{X_T^s},
\end{equation} where $c(T,\lambda)\rightarrow 0^+$ as $T\rightarrow 0^+$ and it may depend on also other fixed parameters of the model. It follows that $\Phi$ is a contraction on $\overline{B_{R_T}(0)}$ for small enough $T$, therefore has a unique fixed point $u\in \overline{B_{R_T}(0)}$, establishing local existence of a solution $u\in X_T^s$.  

Uniqueness in $X_T^s$ (not only in $\overline{B_{R_T}(0)}$) can be obtained through energy method. Assuming existence of two solutions $u_1,u_2\in X_T^s$ and setting $u=u_1-u_2$, we deduce that $u$ solves \eqref{linear} with interior source $f=\kappa(|u_1|^{\lambda-1}u_1-|u_2|^{\lambda-1}u_2)$ and zero initial and boundary data (i.e., $u_0\equiv 0$, $g_0\equiv g_1\equiv 0$). Upon multiplication of the main equation with $\bar{u}$, integrating in $x$, taking the imaginary parts, and using the embedding $H^{s}_x(\mathbb{{R}_+})\hookrightarrow L^{\infty}(\mathbb{{R}_+})$ for $s>\frac12$, we get
\begin{equation}
\begin{split}
    \frac{1}{2}\frac{d}{dt}\|u(t)\|_{L^2(\mathbb{R}_+)}^2 &=\int_0^\infty \bar{u}fdx \le c\int_0^\infty (|u_1|^{\lambda-1}+|u_2|^{\lambda-1})|u_1-u_2|^2dx\\
    &\le c\left(\|u_1(t)\|_{L^\infty_x(\mathbb{R}_+)}^{\lambda-1}+\|u_2(t)\|_{L^\infty_x(\mathbb{R}_+)}^{\lambda-1}\right)\|u(t)\|_{L^2_x(\mathbb{R}_+)}^2\\
    &\le c\left(\|u_1\|_{X_T^s}^{\lambda-1}+\|u_2\|_{X_T^s}^{\lambda-1}\right)\|u(t)\|_{L^2_x(\mathbb{R}_+)}^2.
    \end{split}
\end{equation}
It follows $$\|u(t)\|_{L^2_x(\mathbb{R}_+)}\le \|u_0\|_{L^2_x(\mathbb{R}_+)} e^{c\left(\|u_1\|_{X_T^s}^{\lambda-1}+\|u_2\|_{X_T^s}^{\lambda-1}\right)t},\quad t\in (0,T),$$ from which we obtain $u\equiv 0$ (equivalently, $u_1\equiv u_2$) in view of $u_0\equiv 0$.

To prove continuous dependence on data, let us first fix an initial-boundary data triple: $d:=(u_{0},g_{0},g_{1})$ and let $T$ be the life-span (not necessarily maximal) of the associated local solution satisfying $c_3(T)R_T^{\lambda-1}<\frac{1}{2}$ in \eqref{c3RT} and $c(T,\lambda)R_T^{\lambda-1}<1$ in \eqref{clambdaTRT}. Note that $R_T$ is continuous when considered to be a map of data. It follows that there is $\rho>0$ (depending on $T$ and $d$) such that if $\tilde{d}=(\widetilde{u_0},\tilde{g_0},\tilde{g_1})$ is any other data with $$\|u_0-\widetilde{u_0}\|_{H^s_x(\mathbb{R}_+)}+\|g_0-\widetilde{g_0}\|_{H^{\frac{s+1}{3}}_t(\mathbb{R}_+)}+\|g_1-\widetilde{g_1}\|_{H^{\frac{s}{3}}_t(\mathbb{R}_+)}<\rho,$$ then the solution associated with $\tilde{d}$ also exists on the  same interval $[0,T]$.  Let $u_1,u_2\in X_T^s$ be the solutions associated with data $d,\tilde{d}$, respectively. Then, $u_1-u_2$ solves \eqref{linear} with initial-boundary data $(u_0-\widetilde{u_0},g_0-\tilde{g_0},g_1-\tilde{g_1})$ and interior source $f=\kappa(|u_1|^{\lambda-1}u_1-|{u}_2|^{\lambda-1}{u}_2)$. Using arguments similar to those used for \eqref{Phiinvariance} and \eqref{clambdaTRT}, we obtain
\begin{equation*}
    \|u_1-u_2\|_{X_T^s}\le c_1(T)\|u_0-\widetilde{u_0}\|_{H^s_x(\mathbb{R}_+)} +c_2(T)\left(\|g_0-\widetilde{g_0}\|_{H_t^{\frac{s+1}{3}}(0,T)}+\|g_1-\widetilde{g_1}\|_{H_t^{\frac{s}{3}}(0,T)}\right) +c(T,\lambda)R_T^{\lambda-1}\|u_1-u_2\|_{X_T^s}.
\end{equation*}
Since $c(T,\lambda)R_T^{\lambda-1}<1$, the last term on the right-hand side of the above estimate can be absorbed, with the resulting estimate implying continuous dependence on the data.

\subsection{Low regularity setting - Proof of Theorem \ref{LowRegThem}}
Here, we consider the case $s\in [0,\frac12)$ and $\lambda\ge 2$. For $T>0$, we set $Y_T^s\equiv X_T^s\cap L^\mu_t(0,T;H^{s,r}_x(\mathbb{R}_+))$, where $\mu=\frac{6\lambda}{(1-2s)(\lambda-1)}$ and $r=\frac{2\lambda}{1+2(\lambda-1)s}$ form an admissible pair.  We perform a different decomposition of the forced linear problem that results in the following solution map instead of the one leading to \eqref{PhiDef}: 
%
\begin{equation}\label{lindec}
S_b[u_0, g_0, g_1; f](x, t) = 
S_b[u_0, g_0, g_1; 0](x, t)
+
\int_0^t S_b[f(\cdot, t'), 0, 0; 0](x, t-t') dt',
\end{equation}
where $S_b$ denotes the solution operator to the full forced linear initial-boundary value problem.
The homogeneous component on the right side can be further decomposed as in the case of high regularity:
\begin{equation}
S_b[u_0, g_0, g_1; 0] 
=
y\big|_{Q_T}
+
S_b[0, \psi_0, \psi_1; 0],
\end{equation}
where $y$ is as in \eqref{PhiDef} and 
$\psi_j = g_j -\partial_x^j y|_{x=0}$, $j=0, 1$. Thus, estimating  via the Sobolev estimates of Theorems \ref{reduced} and \ref{cauchylemma} as well as the Strichartz estimates of Theorems \ref{homStr} and \ref{reduced-s}, we have
\begin{equation}\label{homs}
\left\| S_b[u_0, g_0, g_1; 0] \right\|_{Y_T^s}
\lesssim
\left\|u_0\right\|_{H_x^s(\mathbb R_+)} 
+
\left\|g_0\right\|_{H_t^{\frac{s+1}{3}}(0, T)} 
+
\left\|g_1\right\|_{H_t^{\frac{s}{3}}(0, T)}.
\end{equation}
As this estimate \textit{does not require the prescription of compatibility conditions} due to the fact that $s<\frac 12$ (and so we are below the continuity threshold), we may also employ it for the second term on the right side of \eqref{lindec}. In particular, by Minkowski's inequality,
\begin{align}
\left\|\int_0^t S_b[f(\cdot, t'), 0, 0; 0](x, t-t') dt'\right\|_{X_T^s}
&\leq
\sup_{t\in [0, T]} \int_0^t \left\|S_b[f(\cdot, t'), 0, 0; 0](x, t-t')\right\|_{H_x^s(\mathbb R_+)} dt'
\nonumber\\
&\lesssim
\sup_{t\in [0, T]} \int_0^t \left\|f(\cdot, t')\right\|_{H_x^s(\mathbb R_+)} dt'
\leq
\left\|f\right\|_{L_t^1(0, T; H_x^s(\mathbb R_+))}.
\end{align}
Moreover, since for any $1 \le p \le \infty$ Minkowski's integral inequality implies 
\begin{equation}
\norm{\int_0^t \varphi(t, t') dt'}_{L^p_t (0, T)}
\le \int_0^T \norm{\varphi(t')}_{L^p_t (t', T)} dt',
\end{equation}
we have
\begin{align}
\left\|\int_0^t S_b[f(\cdot, t'), 0, 0; 0](x, t-t') dt'\right\|_{L_t^\mu(0, T; H_x^s(\mathbb R_+))}
&\leq
\int_0^T \left\|S_b[f(\cdot, t'), 0, 0; 0](x, t-t')\right\|_{L_t^\mu(t', T; H_x^s(\mathbb R_+))} dt'
\nonumber\\
&\lesssim
\int_0^T \left\|f(\cdot, t')\right\|_{H_x^s(\mathbb R_+)} dt'
=
\left\|f\right\|_{L_t^1(0, T; H_x^s(\mathbb R_+))}
\end{align}
after observing that, although the norm inside the second integral is over the interval $(t', T)$ instead of the usual interval $(0, T)$, the Strichartz estimate of Theorem \ref{homStr} remains unaffected (indeed, recall that the Strichartz estimate for the homogeneous Cauchy problem is in fact established over $t\in \mathbb R$). 
Therefore, we overall deduce
\begin{equation}\label{yst}
\left\| S_b[u_0, g_0, g_1; f] \right\|_{Y_T^s}
\lesssim
\left\|u_0\right\|_{H^s(\mathbb R_+)} 
+
\left\|g_0\right\|_{H^{\frac{s+1}{3}}(0, T)} 
+
\left\|g_1\right\|_{H^{\frac{s}{3}}(0, T)}
+
\left\|f\right\|_{L_t^1(0, T; H_x^s(\mathbb R_+))}.
\end{equation}
In turn, the iteration map
\begin{equation}
\Phi(u) := S_b[u_0, g_0, g_1; \kappa |u|^{\lambda-1} u]
\end{equation}
admits the estimate
\begin{equation}\label{yst2}
\left\| \Phi(u) \right\|_{Y_T^s}
\lesssim
\left\|u_0\right\|_{H^s(\mathbb R_+)} 
+
\left\|g_0\right\|_{H^{\frac{s+1}{3}}(0, T)} 
+
\left\|g_1\right\|_{H^{\frac{s}{3}}(0, T)}
+
|\kappa| \left\||u|^{\lambda-1} u\right\|_{L_t^1(0, T; H_x^s(\mathbb R_+))}.
\end{equation}

We recall the following basic lemma to handle the nonlinearity:

\begin{lemma}[\cite{MMO26}]\label{L1estimates}
    Suppose $0 \le s < \frac12$, $2 \le \lambda \le \frac{7-2s}{1-2s}$, and
$$
(\mu,r)=\left(\frac{6\lambda}{(1-2s)(\lambda-1)},\frac{2\lambda}{1+2(\lambda-1)s}\right).
$$
Then,
$$
\bigl\||\varphi|^{\lambda-1}\varphi\bigr\|_{L_t^1((0,T);H_x^s(\mathbb{R}))}
\le c(s,\lambda)T^{\frac{\mu-\lambda}{\mu}}
\|\varphi\|_{L_t^\mu((0,T);H_x^{s,r}(\mathbb{R}))}^{\lambda},
$$
and
$$
\bigl\||\varphi|^{\lambda-1}\varphi-|\psi|^{\lambda-1}\psi\bigr\|_{L_t^1((0,T);H_x^s(\mathbb{R}))}
\le c(s,\lambda)T^{\frac{\mu-\lambda}{\mu}}
\left(
\|\varphi\|_{L_t^\mu((0,T);H_x^{s,r}(\mathbb{R}))}^{\lambda-1}
+
\|\psi\|_{L_t^\mu((0,T);H_x^{s,r}(\mathbb{R}))}^{\lambda-1}
\right)
\|\varphi-\psi\|_{L_t^\mu((0,T);H_x^{s,r}(\mathbb{R}))}.
$$
\end{lemma}

It should be noted that although the above lemma uses the whole line for the spatial domain, its estimates are also true for the half-line since the zero-extension operator is bounded in the regime $0\le s<\frac12$. Thus, \eqref{yst2} takes the form 
\begin{equation}\label{lowregest02}
    \|\Phi(u)\|_{Y_T^s}\le c_1(T)\|u_0\|_{H^s_x(\mathbb{R}_+)} +c_2(T)\left(\|g_0\|_{H_t^{\frac{s+1}{3}}(0,T)}+\|g_1\|_{H_t^{\frac{s}{3}}(0,T)}\right) +c_3(T)T^{1-\frac{\lambda}{\mu}}\|u\|_{Y_T^s}^\lambda,
\end{equation} where now $c_3$ also depends on $\lambda,s,\kappa$. Moreover, in the above estimate $c_i(T)$, $i=1,2,3$, remain bounded as $T\rightarrow 0^+$.  Assume without loss of generality that either $u_0$ or else at least one of the $g_i$'s is nonzero (otherwise the zero solution already satisfies the nonlinear ibvp). Set $R_T$ as in \eqref{defrT}.  Then, for any $u\in \overline{B_{R_T}(0)}\subset Y_T^s$, \begin{equation}\label{c3RT}
     \|\Phi(u)\|_{Y_T^s} \le \frac{1}{2}R_T +c_3(T)T^{1-\frac{\lambda}{\mu}}R_T^\lambda\le R_T
\end{equation} provided $T$ is small enough and $\lambda<\frac{7-2s}{1-2s}$. That is, $\Phi$ is invariant on $\overline{B_{R_T}(0)}$. In the critical case $\lambda=\frac{7-2s}{1-2s}$, we have $1-\frac{\lambda}{\mu}=0$, therefore taking $T$ will not help on its own.  In this case, we in addition assume initial datum $u_0$ to be small for invariance to hold, noting that $\|g_0\|_{H_t^{\frac{s+1}{3}}(0,T)}+\|g_1\|_{H_t^{\frac{s}{3}}(0,T)}$ can be already made small by taking $T$ small without imposing an extra smallness assumption on the global norms of $g_0$ and $g_1$.

Again in view of Lemma \ref{L1estimates} and the linear estimates of the manuscript, regarding the differences, we get
\begin{equation}\label{clambdaTRTlow}
     \|\Phi(u_1)-\Phi(u_2)\|_{Y_T^s} \le c(T,\lambda)T^{1-\frac{\lambda}{\mu}}R_T^{\lambda-1}\|u_1-u_2\|_{Y_T^s},
\end{equation} where $c(T,\lambda)$ remains bounded as $T\rightarrow 0^+$ and it may depend on also other fixed parameters of the problem. It follows that $\Phi$ is a contraction on $\overline{B_{R_T}(0)}$ for small enough $T$ (provided $\lambda<\frac{7-2s}{1-2s}$ or $u_0$ small when $\lambda=\frac{7-2s}{1-2s}$), therefore has a unique fixed point $u\in \overline{B_{R_T}(0)}$, establishing local existence of a solution $u\in Y_T^s$.  Uniqueness of the local solution in $Y_T^s$ can be proven by adapting the uniqueness arguments in \cite[Section 3.3]{amo2024} to two boundary condition and the continuous dependence on data can be done as in the high regularity setting.

\bibliographystyle{amsalpha}

\def\CMbib{C:/Users/Chris/Documents/School/Research/referencesto.bib}
\def\DMbib{referencesto.bib}

\IfFileExists{\CMbib}{\bibliography{\CMbib}}{
	\IfFileExists{\DMbib}{\bibliography{\DMbib}}{
}}
\end{document}